\documentclass[12pt,leqno]{article}

\usepackage{amsmath,amsthm,amsfonts,latexsym,amscd,amssymb}
\numberwithin{equation}{section}
\input xypic
\def\qed{{\hbadness=10000\hfill\ \vbox{\hrule height.09ex
   \hbox{\vrule width.09ex height1.55ex depth.2ex \kern1.8ex
   \vrule width.09ex height1.55ex depth.2ex}\hrule height.09ex}\break
   \bigskip}}
\newtheorem{theorem}{Theorem}[section]
\newtheorem{lemma}{Lemma}[section]

\newtheorem{proposition}{Proposition}[section]

\theoremstyle{definition}

\theoremstyle{remark}

\begin{document}

\linespread{1}\title{\textbf{Finsler space subjected to a Kropina change with an \textsl{h}-vector}}

\author{M.$\,$K. \textsc{Gupta}\thanks{Corresponding author. E-mail: mkgiaps@gmail.com}~\\
\normalsize{Department of Pure $\&$ Applied Mathematics}\\
\normalsize{Guru Ghasidas Vishwavidyalaya}\\
\normalsize{Bilaspur (C.G.), India}\\ \\
P.$\,$N.$\,\,$\textsc{Pandey}\\
\normalsize{Department of Mathematics}\\
\normalsize{University of Allahabad}\\
\normalsize{Allahabad, India}}
\date{}
\maketitle

\linespread{1.3}\begin{abstract} In this paper, we discuss the Finsler spaces $(M^n,L)$ and $(M^n,\,^{*}L)$, where $^{*}L(x,y)$ is obtained from $L(x,y)$ by Kropina change $^{*}L(x,y)=\frac{L^2(x,y)}{b_i(x,y)\,y^i}$ and $b^{}_{i}(x,y)$ is an \textsl{h}-vector in $(M^n,L)$. We find the necessary and sufficient condition when the Cartan connection coefficients for both spaces $(M^n,L)$ and $(M^n,\,^{*}L)$ are the same. We also find the necessary and sufficient condition for Kropina change with an \textsl{h}-vector to be projective.
\end{abstract}

\section{Introduction}
~~~~%Let $F^{n}=(M^{n},L)$ be an $n$-dimensional Finsler space equipped with the fundamental function $L(x,y)$. The normalized supporting element, the metric tensor, the angular metric tensor and Cartan tensor are defined by
%\[l_{i}=\dot{\partial_{i}}L\,,\quad g_{ij}=\frac{1}{2}\,\dot{\partial_{i}}\,\dot{\partial_{j}}L^{2}\,,\quad h_{ij}=L\,\dot{\partial_{i}}\,\dot{\partial_{j}}L \quad \textrm{and}\quad C_{ijk}=\frac{1}{2}\,\dot{\partial_{k}}\,g_{ij}\quad \textrm{respectively},\]
%where $\dot{\partial_{i}}=\partial/\partial y^{i}$. The Cartan connection in $F^{n}$ is given as $C\Gamma=(F^{\,i}_{jk},N^{\,i}_{j},C^{\,i}_{jk})$. The \textsl{h}- and \textsl{v}-covariant derivatives of a covariant vector $X_{i}(x,y)$ with respect to the Cartan connection are respectively given by
%\begin{equation}X_{i|j}=\partial_{j}X_{i}-(\dot{\partial}_{h}X_{i})N^{h}_{j}-F^{\,r}_{ij}X_{r}\,,\end{equation}
%\begin{equation}\textrm{and}\qquad\qquad X_{i}|_{j}=\dot{\partial}_{j}X_{i}-C^{\,r}_{ij}X_{r}\,,~~~~~~~~~~~~~~~~\end{equation}
%where $\partial_{j}=\partial/\partial x^{j}$.
In 1984, C. Shibata \cite{Sh84} dealt with a change of Finsler metric which is called a $\beta$-change of metric. A remarkable class of $\beta$-change is Kropina change $^{*}L(x,y)=\frac{L^2(x,y)}{b_i(x)\,y^i}$. If $L(x,y)$ is a metric function of a Riemannian space then $^{*}L(x,y)$ reduces to the metric function of a Kropina space. Kropina metric was first introduced by L. Berwald in connection with a two-dimensional Finsler space with rectilinear extremals and was investigated by V.K. Kropina \cite{Kr59,Kr61}. Kropina metric is simplest non-trivial Finsler metric having many interesting applications in Physics, Electron optics with a magnetic field, plants study for fungal fusion hypothesis, dissipative mechanics and irreversible thermodynamics \cite{An93,An12,As85,In87}. In 1978, C. Shibata \cite{Sh78} studied some basic local geometric properties of Kropina spaces. In 1991, M. Matsumoto obtained a set of necessary and sufficient conditions for a Kropina space to be of constant curvature \cite{Ma91}.

H. Izumi \cite{Iz80} while studying the conformal transformation of Finsler spaces, introduced the concept of \textsl{h}-vector $b_{i}$, which is \textsl{v}-covariant constant with respect to the Cartan connection and satisfies $L\,C^{h}_{ij}\,b_{h}=\rho\, h_{ij}\,,\,$ where $\rho$ is a non-zero scalar function, $C^{h}_{ij}$ are components of Cartan tensor and $h_{ij}$ are components of angular metric tensor. Thus if $b_{i}$ is an \textsl{h}-vector then
\begin{equation}(i)\,\,b_{i}|_{k}=0\,,\qquad\qquad\qquad(ii)\,\, L\,C^{h}_{ij}\,b_{h}=\rho\,h_{ij}\,.\end{equation}
This gives
\begin{equation}L\,\dot{\partial_{j}}b_{i}=\rho \,h_{ij}\,. \end{equation}
Since $\rho\neq 0$ and $h_{ij}\neq0$, the \textsl{h}-vector $b_{i}$ depends not only on positional coordinates but also on directional arguments. Izumi \cite{Iz80} showed that $\rho$ is independent of directional arguments. M. Matsumoto \cite{Ma74} discussed the Cartan connection of Randers change of Finsler metric, while B. N. Prasad \cite{Pr90} obtained the Cartan connection of $(M^n,\,^{*}L)$ where $^{*}L(x,y)$ is given by $^{*}L(x,y)=L(x,y)+b_i(x,y)\,y^i,$ and $b_i(x,y)$ is an \textsl{h}-vector. Present authors \cite{GuPa09,GuPa08} discussed the hypersurface of a Finsler space whose metric is given by certain transformations with an \textsl{h}-vector.
In this paper we obtain the relation between the Cartan connections of $F^{n}=(M^{n},L)$ and $^{*}F^{n}=(M^{n},\,^{*}L)$ where $^{*}L(x,y)$ is obtained by the transformation
\begin{equation}^{*}L(x,y)=\frac{L^2(x,y)}{b_i(x,y)\,y^i}\,,\end{equation}
and $b_{i}(x,y)$ is an \textsl{h}-vector in $(M^{n},L)$.
\\The paper is organized as follows: In section 2, we study how the fundamental metric tensor and the Cartan tensor change by Kropina change with an \textsl{h}-vector. The relation between the Cartan connection coefficients of both spaces is obtained in the section 3 and we find the necessary and sufficient condition when these connection coefficients are the same. In section 4, we find the necessary and sufficient condition for the Kropina change with an \textsl{h}-vector to be projective.

%The terminologies and notations are referred to Matsumoto \cite{Ma86}.

\section{The Finsler space $^{*}F^{n}=(M^{n},\,^{*}L)$}
~~~~Let $F^{n}=(M^{n},L)$ be an $n$-dimensional Finsler space equipped with the fundamental function $L(x,y)$. We consider a change of Finsler metric $^{*}L(x,y)$ which is defined by (1.3) and have another Finsler space $^{*}F^{n}=(M^{n},\,^{*}L)$. If we denote $b^{}_{i}\,y^i$ by $\beta$ then the indicatory property of $h_{ij}$ yields $\dot{\partial_{i}}\beta=b_i$. Throughout this paper, the geometric objects associated with $^{*}F^{n}$ will be asterisked. We shall use the notation $L_{i}=\dot{\partial_{i}}L=l_i\,,\,\,L_{ij}=\dot{\partial_{i}}\,\dot{\partial_{j}}L\,,\,\,L_{ijk}=\dot{\partial_{k}}\,L_{ij}\,,\ldots$ etc.
From (1.3), we get
\begin{equation}^{*}L_{i}=2\,\tau \,L_{i}-\tau^2\,b_{i}\,,\end{equation}
\begin{equation}^{*}L_{ij}=(2\tau-\rho\tau^2)\,L_{ij}+\frac{2\,\tau^2}{\beta}\,m_i\,m_j\,,\end{equation}
\begin{equation}\begin{split}^{*}L_{ijk}=&(2\tau-\rho\tau^2)\,L_{ijk}+\frac{2\,\tau}{\beta}\,(\rho\tau-1)\,(m_iL_{jk}+m_jL_{ik}+m_kL_{ij})\\
&-\frac{2\tau^2}{L\beta}(m_im_jl_k+m_jm_kl_i+m_km_il_j)-\frac{6\tau^2}{\beta^2}m_im_jm_k\,,\end{split}\end{equation}
where $\tau=L/\beta\,,\,\,m_{i}=b_{i}-\frac{1}{\tau}l_{i}$. The normalized supporting element, the metric tensor and Cartan tensor of $^{*}F^{n}$ are obtained as
\begin{equation}^{*}l_{i}=2\,\tau l_{i}-\tau^2\,b_{i}\,,\end{equation}
\begin{equation}^{*}g^{}_{ij}=(2\tau^2-\rho\tau^3)\,g^{}_{ij}+3\tau^4\,b_{i}b_{j}-4\tau^3(l_{i}b_{j}+b_{i}l_{j})+(4\tau^2+\rho\tau^3)l_{i}l_{j}\,, \end{equation}
\begin{equation}^{*}C_{ijk}=(2\tau^2-\rho\tau^3)\,C_{ijk}-\frac{\tau^2}{2\,\beta}(4-3\rho\tau)(h_{ij}m_{k}+h_{jk}m_{i}+h_{ki}m_{j})-\frac{6\tau^2}{\beta}m_im_jm_k\,.\end{equation}

\noindent For the computation of the inverse metric tensor, we use the following lemma \cite{Ma72}:
\begin{lemma}Let $(m^{}_{ij})$ be a non-singular matrix and $l^{}_{ij}=m^{}_{ij}+n^{}_in^{}_j$. The elements $l^{ij}_{}$ of the inverse matrix and the determinant of the matrix $(l^{}_{ij})$ are given by
\[l^{ij}_{}=m^{ij}_{}-(1+n^{}_{k}n^{k}_{})^{-1}\,n^in^j,\quad det(l_{ij})=(1+n^{}_{k}n^{k}_{})\,det(m_{ij})\]
respectively, where $m^{ij}_{}$ are elements of the inverse matrix of $(m^{}_{ij})$ and $n^k=m^{ki}_{}n_i$.
\end{lemma}

The inverse metric tensor of $^{*}F^{n}$ is derived as follows:
\begin{equation}\begin{split}^{*}g^{ij}=(2\tau^2-\rho\tau^3)^{-1}&\Big[g^{ij}-\frac{2\tau}{2b^2\tau-\rho}b^i\,b^j+\frac{4-\rho\tau}{2b^2\tau-\rho}(l^{i}\,b^{j}+b^{i}\,l^{j})\\[5mm] &-\frac{3\rho b^2\tau^3-\rho^2\tau^2-4b^2\tau^2-2\rho\tau+8}{\tau(2b^2\tau-\rho)}l^i\,l^j\Big] \end{split}\end{equation}
where $b$ is the magnitude of the vector $b^{i}=g^{ij}b_{j}$.\\
From (2.6) and (2.7), we get
\begin{equation}\begin{split}^{*}C^{h}_{ij}=C^{h}_{ij}-\frac{(4-3\rho\tau)\tau}{2L(2-\rho\tau)}\big(h^{}_{ij}\,m^h+&h^h_j\,m_i+h^h_i\,m_j\big)-\frac{6\tau}{L(2-\rho\tau)}m_im_jm^h\\[3mm]
+\frac{2\tau\,b^h-(4-\rho\tau)l^h}{L(2-\rho\tau)(2b^2\tau-\rho)}\Big[&h^{}_{ij}\Big\{\frac{1}{2}m^2\tau(4-3\rho\tau)-\rho(2-\rho\tau)\Big\}\\
&+m_im_j\Big\{6\tau\,m^2+\tau(4-3\rho\tau)\Big\}\Big].\end{split}\end{equation}

\section{Cartan connection of the space $^{*}F^{n}$}
~~~~Let $C\,^{*}\Gamma=(\,^{*}F^{\,i}_{jk},\,^{*}N^{\,i}_{j},\,^{*}C^{\,i}_{jk})$ be the Cartan connection for the space $^{*}F^{n}=(M^{n},\,^{*}L)$. Since for a Cartan connection $L^{}_{i|j}=0$, we obtain
\begin{equation}\partial_j L_i=L_{ir}N^r_j+L_rF^{\,r}_{ij}\,.\end{equation}
Differentiating (2.1) with respect to $x^j$, we get
\begin{equation}\partial_j\,^{*}L_i=2\tau\,\partial_j L_i+2L_i\,\partial_j \tau-\tau^2\,\partial_j b_i-2\tau\,b_i\,\partial_j \tau\,.\end{equation}
This equation may be written in tensorial form as
\begin{equation}\begin{split} ^{*}L_{ir}\,^{*}N^r_j+\,^{*}L_r\,^{*}F^{\,r}_{ij}&=2\tau\,(L_{ir}N^r_j+L_rF^{\,r}_{ij})+\frac{2L_i}{\beta}(N^r_j\,L_r-\tau\beta_j-\tau N^r_jb_r)\\
&-\tau^2(b^{}_{i|j}+\rho L_{ir}N^{r}_{j}+b_rF^{\,r}_{ij})-\frac{2\tau}{\beta}\,b_i(N^r_j\,L_r-\tau\beta_j-\tau N^r_jb_r)\,,\end{split}\end{equation}
where $\beta^{}_{j}=\beta^{}_{|j}$. If we put
\begin{equation}^{*}F^{\,i}_{jk}=F^{\,i}_{jk}+D^{\,i}_{jk}\,,\end{equation}
then in view of (2.2), equation (3.3) may be written as
\begin{equation}(2\tau l_r-\tau^2b_r)\,D^{r}_{ij}+\Big\{(2\tau-\rho\tau^2)L_{ir}+\frac{2\tau^2}{\beta}\,m_im_r\Big\}D^{r}_{0j} =\frac{2\tau^2}{\beta}\,m_i\beta^{}_{j}-\tau^2b^{}_{i|j}\,, \end{equation}
where the subscript `$0$' denote the contraction by $y^i$.

In order to find the difference tensor $D^{i}_{jk}$, we construct supplementary equations to (3.5). From (2.2), we obtain
\begin{equation}\begin{split}\partial_k\,^{*}L_{ij}=&(2\tau-\rho\tau^2)\,\partial_k L_{ij}+L_{ij}(2\partial_k \tau-2\rho\tau\,\partial_k \tau-\tau^2\partial_k\rho)\\[3mm]
&+\frac{2\,\tau^2}{\beta}\,m_i\,\partial_k m_j+\frac{2\,\tau^2}{\beta}\,m_j\,\partial_k m_i+2m_im_j\frac{1}{\beta^2}\big(2\tau\beta\,\partial_k \tau-\tau^2\,\partial_k \beta\big)\,.\end{split}\end{equation}
From $L^{}_{ij|k}=0$, equation (3.6) is written in the form
\begin{equation*}\begin{split} ^{*}L_{ijr}\,^{*}N^r_k+\,^{*}L_{rj}\,^{*}F^{\,r}_{ik}+\,^{*}L_{ir}\,^{*}F^{\,r}_{jk} =&(2\tau-\rho\tau^2)\,\{L^{}_{ijr}N^r_k+L^{}_{rj}F^{\,r}_{ik}+L^{}_{ir}F^{\,r}_{jk}\}\\[2mm]
&+L_{ij}\Big\{\frac{2}{\beta}(1-\rho\tau)(-\tau\beta^{}_{k}-\tau\,N^r_km_r)-\tau^2\rho^{}_{k}\Big\}\\
&+\frac{2\tau^2}{\beta}m_{i}\Big\{F^{\,r}_{jk}m_r+N^r_kL_{jr}\Big(\rho-\frac{1}{\tau}\Big)-\frac{1}{\beta\tau}N^r_kl_jm_r\Big\}\\
&+\frac{2\tau^2}{\beta}m_{j}\Big\{F^{\,r}_{ik}m_r+N^r_kL_{ir}\Big(\rho-\frac{1}{\tau}\Big)-\frac{1}{\beta\tau}N^r_kl_im_r\Big\}\\
&+\frac{2}{\beta^2}m_im_j\Big\{2\tau(-\tau\beta^{}_{k}-\tau\,N^r_km_r)-\tau^2(\beta^{}_{k}+N^r_kb_r)\Big\},\end{split}\end{equation*}
where $\rho^{}_{k}=\rho^{}_{|k}=\partial_k\rho$. In view of (2.2), (2.3) and (3.4), above equation is written as
\begin{equation}\begin{split} &(2\tau-\rho\tau^2)\,\{L^{}_{ijr}D^r_{0k}+L^{}_{rj}D^{r}_{ik}+L^{}_{ir}D^{r}_{jk}\}+\frac{2\tau^2}{\beta}m_r(m_jD^{r}_{ik}+m_iD^{r}_{jk})\\[2mm]
&-\frac{2\tau^2}{\beta^2}(m_im_jb_r+m_jm_rb_i+m_rm_ib_j)D^{r}_{0k}+\frac{2\tau}{\beta}\beta^{}_{k}L_{ij}-\frac{2\rho\tau^2}{\beta}\,\beta^{}_{k}L_{ij}\\[2mm]
&+\frac{2\tau^2}{\beta}(\rho-\frac{1}{\tau})(L_{jr}m_i+L_{ir}m_j+L_{ij}m_r)D^{r}_{0k}+\frac{6\tau^2}{\beta^2}\,\beta^{}_{k}m_im_j+\tau^2\rho^{}_{k}L_{ij}=0.\end{split}\end{equation}

\noindent Now we will prove:
\begin{proposition} The difference tensor $D^{i}_{jk}$ is completely determined by the equations (3.5) and (3.7).
\end{proposition}

To prove this, first we will prove a lemma:
\begin{lemma} The system of algebraic equations
\begin{equation*}(i)\,\,^{*}L_{ir}A^r=B_i\,,\qquad\qquad\qquad(ii)\,\,^{*}L_{r}A^r=B\,,\end{equation*}
has a unique solution $A^r$ for given $B$ and $B_i$.
\end{lemma}
\begin{proof} It follows from (2.2) that $(i)$ is written in the form
\begin{equation}\Big\{(2\tau-\rho\tau^2)\frac{1}{L}(g_{ir}-l_il_r)+\frac{2\tau^2}{\beta}m_im_r\Big\}A^{\,r}_{}=B_i.\end{equation}
Contracting by $b^i$, we get
\begin{equation*}\Big\{(2\tau-\rho\tau^2)\frac{1}{L}\Big(b_{r}-\frac{1}{\tau}l_r\Big)+\frac{2\tau^2}{\beta}m^2\,m_r\Big\}A^{\,r}_{}=B^{}_{\beta},\end{equation*}
i.e.
\begin{equation}m_r\,A^{\,r}_{}=B^{}_{\beta}\Big(\frac{2\tau^2\,b^2}{\beta}-\frac{\rho\tau^2}{L}\Big)^{-1},\end{equation}
where the subscript $\beta$ denote the contraction by $b^i$, i.e. $B^{}_{\beta}=B^{}_{i}b^i$.\\
Also from (2.1), equation $(ii)$ is written in the form
\[(2\tau\,l_r-\tau^2b_r)\,A^r=B\]
i.e.
\begin{equation}\tau^2m_rA^r-\tau\,l_rA^r=-B.\end{equation}
Using (3.9) in (3.10), we get
\begin{equation*}l_r\,A^{\,r}_{}=\tau^{-1}B+\tau\,B_{\beta}\,\Big(\frac{2\tau^2\,b^2}{\beta}-\frac{\rho\tau^2}{L}\Big)^{-1}.\end{equation*}
Then (3.8) is written as
\begin{equation*}g_{ir}\,A^{\,r}_{}=\frac{L}{2\tau-\rho\tau^2}B_i+l_i\,\Big\{\tau^{-1}B+\tau\,B_{\beta}\,\Big(\frac{2\tau^2\,b^2}{\beta}-\frac{\rho\tau^2}{L}\Big)^{-1}\Big\}-\frac{2\tau^2}{2-\rho\tau}\,m_i\,B_{\beta}\,\Big(\frac{2\tau^2\,b^2}{\beta}-\frac{\rho\tau^2}{L}\Big)^{-1}.\end{equation*}
This gives
\begin{equation}A^{\,i}_{}=\frac{L}{2\tau-\rho\tau^2}B^i+l^i\,\Big\{\tau^{-1}B+\tau\,B_{\beta}\,\Big(\frac{2\tau^2\,b^2}{\beta}-\frac{\rho\tau^2}{L}\Big)^{-1}\Big\}-\frac{2\tau^2}{2-\rho\tau}\,m^i\,B_{\beta}\,\Big(\frac{2\tau^2\,b^2}{\beta}-\frac{\rho\tau^2}{L}\Big)^{-1},\end{equation}
which is the concrete form of the solution $A^i$.
\end{proof}

We are now in a position to prove the proposition.

\noindent Taking the symmetric and anti-symmetric parts of (3.5), we get
\begin{equation}\begin{split}2(2\tau l_r&-\tau^2b_r)\,D^{r}_{ij}+\Big\{(2\tau-\rho\tau^2)L_{ir}+\frac{2\tau^2}{\beta}\,m_im_r\Big\}D^{r}_{0j}\\
&+\Big\{(2\tau-\rho\tau^2)L_{jr}+\frac{2\tau^2}{\beta}\,m_jm_r\Big\}D^{r}_{0i}=\frac{2\tau^2}{\beta}\,(m_i\beta^{}_{j}+m_j\beta^{}_{i})-2\tau^2E^{}_{ij}\,, \end{split}\end{equation}
and
\begin{equation}\begin{split}\Big\{(2\tau-\rho\tau^2)L_{ir}+&\frac{2\tau^2}{\beta}\,m_im_r\Big\}D^{r}_{0j}-\Big\{(2\tau-\rho\tau^2)L_{jr}+\frac{2\tau^2}{\beta}\,m_jm_r\Big\}D^{r}_{0i}\\ &=\frac{2\tau^2}{\beta}\,(m_i\beta^{}_{j}-m_j\beta^{}_{i})-2\tau^2F^{}_{ij}\,, \end{split}\end{equation}
where we put $2E^{}_{ij}=b^{}_{i|j}+b^{}_{j|i}$ and $2F^{}_{ij}=b^{}_{i|j}-b^{}_{j|i}$.

On the other hand, applying Christoffel process with respected to indices $i,j,k$ in equation (3.7), we get
\begin{equation}\begin{split} &(2\tau-\rho\tau^2)\,\big\{L^{}_{ijr}D^{r}_{0k}+L^{}_{jkr}D^{r}_{0i}-L^{}_{kir}D^r_{0j}\big\}+2D^{r}_{ik}\Big\{(2\tau-\rho\tau^2)L^{}_{rj}+\frac{2\tau^2}{\beta}m_rm_j\Big\}\\[3mm]
&-\frac{2\tau}{\beta}\Big\{\beta^{}_{k}\big((\rho\tau-1)L_{ij}-\frac{3\tau}{\beta}m_im_j\big)+\beta^{}_{i}\big((\rho\tau-1)L_{jk}-\frac{3\tau}{\beta}m_jm_k\big)-\beta^{}_{j}\big((\rho\tau-1)L_{ki}-\frac{3\tau}{\beta}m_km_i\big)\Big\}\\[3mm]
&+\frac{2\tau}{\beta}D^{r}_{0k}\,\mathfrak{S}_{(ijr)}m_i\Big\{(\rho\tau-1)L_{jr}-\frac{\tau}{\beta}m_jb_r\Big\}+\frac{2\tau}{\beta}D^{r}_{0i}\,\mathfrak{S}_{(jkr)}m_j\Big\{(\rho\tau-1)L_{kr}-\frac{\tau}{\beta}m_kb_r\Big\}\\[3mm]
&-\frac{2\tau}{\beta}D^{r}_{0j}\,\mathfrak{S}_{(kir)}m_k\Big\{(\rho\tau-1)L_{ir}-\frac{\tau}{\beta}m_ib_r\Big\}+\tau^2(\rho^{}_{k}L_{ij}+\rho^{}_{i}L_{jk}-\rho^{}_{j}L_{ki})=0,
\end{split}\end{equation}
where $\mathfrak{S}_{(ijk)}$ denote cyclic interchange of indices $i,j,k$ and summation. Contracting (3.12) and (3.13) by $y^j$, we get
\begin{equation}(4\tau\,l^{}_{r}-2\tau^2b^{}_{r})D^{\,r}_{0i}+\Big\{(2\tau-\rho\tau^2)L_{ir}+\frac{2\tau^2}{\beta}m_im_r\Big\}D^{\,r}_{00}=\frac{2\tau^2}{\beta}\beta^{}_{0}m_i-2\tau^2E^{}_{i0},\end{equation}
and
\begin{equation*}\Big\{(2\tau-\rho\tau^2)L_{ir}+\frac{2\tau^2}{\beta}m_im_r\Big\}D^{\,r}_{00}=\frac{2\tau^2}{\beta}\beta^{}_{0}m_i-2\tau^2F^{}_{i0},\end{equation*}
i.e.
\begin{equation} ^{*}L_{ir}D^{\,r}_{00}=\frac{2\tau^2}{\beta}\beta^{}_{0}m_i-2\tau^2F^{}_{i0},\end{equation}
where $\beta^{}_{0}=\beta^{}_{j}y^j$. Similarly contraction of (3.14) by $y^k$ gives
\begin{equation}\begin{split} &(2\tau-\rho\tau^2)\,\big\{L^{}_{ijr}D^{r}_{00}-L^{}_{jr}D^{r}_{0i}+L^{}_{ir}D^r_{0j}\big\}+2D^{r}_{0i}\big\{(2\tau-\rho\tau^2)L^{}_{rj}+\frac{2\tau^2}{\beta}m_rm_j\big\}\\[2mm]
&+\frac{2\tau}{\beta}D^{r}_{00}\,\mathfrak{S}_{(ijr)}m_i\big\{(\rho\tau-1)L_{jr}-\frac{\tau}{\beta}m_jb_r\big\}-\frac{2\tau^2}{\beta}D^{r}_{0i}\,m_jm_r+\frac{2\tau^2}{\beta}D^{r}_{0j}\,m_im_r\\[2mm]
&-\frac{2\tau}{\beta}\beta^{}_{0}\Big((\rho\tau-1)L_{ij}-\frac{3\tau}{\beta}m_im_j\Big)+\tau^2\rho^{}_{0}L_{ij}=0\,,
\end{split}\end{equation}
Contraction of (3.15) by $y^i$ gives
\begin{equation*}(2\tau\,l^{}_{r}-\tau^2b^{}_{r})D^{\,r}_{00}=-\tau^2E^{}_{00},\end{equation*}\\[-10mm]
i.e.\\[-10mm]
\begin{equation} ^{*}L^{}_{r}D^{\,r}_{00}=E^{}_{00}.\end{equation}

We can apply Lemma 3.1 to equations (3.16) and (3.18) to obtain
\begin{equation}\begin{split}D^{\,i}_{00}=&l^i\,\tau\,\Big\{\Big(\frac{2}{\beta}\beta^{}_{0}m^2-2F^{}_{\beta 0}\Big)\Big(\frac{2}{\beta}b^2-\frac{\rho}{L}\Big)^{-1}-E^{}_{00}\Big\}\\
&-\frac{2\tau^2}{2-\rho\tau}\Big(\frac{2}{\beta}\beta^{}_{0}m^2-2F^{}_{\beta 0}\Big)\Big(\frac{2}{\beta}b^2-\frac{\rho}{L}\Big)^{-1}\,m^i+\frac{2L\tau}{2-\rho\tau}\Big(\frac{1}{\beta}\,\beta^{}_{0}\,m^i-F^{\,i}_{0}\Big),\end{split}\end{equation}
where $F^{i}_{0}=g^{ij}F_{j0}$. Also note that $E_{00}=E_{ij}\,y^i\,y^j=b_{i|j}\,y^i\,y^j=(b_{i}\,y^i)_{|j}y^j=\beta^{}_{|0}=\beta^{}_{0}$.

\noindent Now adding (3.13) and (3.17), we obtain
\begin{equation*} D^{\,r}_{0j}\Big\{(2\tau-\rho\tau^2)L_{ir}+\frac{2\tau^2}{\beta}m_im_r\Big\}=G_{ij}\,,\end{equation*}
i.e.
\begin{equation} ^{*}L_{ir}\,D^{\,r}_{0j}=G_{ij}\,,\end{equation}
where we put
\begin{equation}\begin{split}G_{ij}&=\frac{\tau^2}{\beta}(m_i\beta_j-m_j\beta_i)-\tau^2\,F_{ij}-\frac{1}{2}(2\tau-\rho\tau^2)L_{ijr}D^{\,r}_{00}-\frac{\tau^2}{2}\rho^{}_{0}L_{ij}\\[3mm]
&-\frac{\tau}{\beta}D^{\,r}_{00}\,\mathfrak{S}_{(ijr)}m_i\Big\{(\rho\tau-1)L_{jr}-\frac{\tau}{\beta}m_jb_r\Big\}+\frac{\tau}{\beta}\beta^{}_{0}\Big\{(\rho\tau-1)L_{ij}-\frac{3\tau}{\beta}m_im_j\Big\}\,.\end{split}\end{equation}\\[-3mm]
The equation (3.15) is written in the form
\begin{equation*}(2\tau\,l_r-\tau^2\,b_r)D^{\,r}_{0j}=G_{j}\,,\end{equation*}\\[-11mm]
i.e.\\[-11mm]
\begin{equation}^{*}L^{}_{r}D^{\,r}_{0j}=-\frac{1}{\tau^2}\,G_{j}\,,\end{equation}
where
\[G_{j}=\frac{\tau^2}{\beta}\,\beta^{}_{0}m_j-\tau^2E_{j0}-\Big\{\frac{1}{2}(2\tau-\rho\tau^2)L_{jr}+\frac{\tau^2}{\beta}m_jm_r\Big\}\,D^{\,r}_{00}.\]
In view of (3.16), $G_j$ are written as
\begin{equation}G_{j}=\tau^2(F_{j0}-E_{j0})\,.\end{equation}
Applying Lemma 3.1 to equations (3.20) and (3.22) to obtain 
\begin{equation}D^{\,i}_{0j}=l^i\Big\{\frac{1}{\tau}\Big(\frac{2}{\beta}b^2-\frac{\rho}{L}\Big)^{-1}G_{\beta j}+\frac{1}{\tau}G^{}_{j}\Big\}-\frac{2\,m^i}{2-\rho\tau}\Big(\frac{2}{\beta}b^2-\frac{\rho}{L}\Big)^{-1}G_{\beta j}+\frac{L}{2\tau-\rho\tau^2}G^{\,i}_{j}\,, \end{equation}
where $G^{\,i}_j=g^{ik}G_{kj}$.

Finally we solve (3.12) and (3.14) for $D^{\,i}_{jk}$. These equations may be written as
\begin{equation} ^{*}L^{}_{rj}\,D^{r}_{ik}=H^{}_{jik}\,,\end{equation}
and
\begin{equation} ^{*}L^{}_{r}\,D^{r}_{ik}=H^{}_{ik}\,,\end{equation}
where
\begin{equation}\begin{split} H^{}_{jik}&=\frac{(\rho\tau^2-2\tau)}{2}\,\big\{L^{}_{ijr}D^{r}_{0k}+L^{}_{jkr}D^{r}_{0i}-L^{}_{kir}D^r_{0j}\big\}\\
&-\frac{\tau}{\beta}D^{r}_{0k}\,\mathfrak{S}_{(ijr)}m_i\Big\{(\rho\tau-1)L_{jr}-\frac{\tau}{\beta}m_jb_r\Big\}-\frac{\tau}{\beta}D^{r}_{0i}\,\mathfrak{S}_{(jkr)}m_j\Big\{(\rho\tau-1)L_{kr}-\frac{\tau}{\beta}m_kb_r\Big\}\\
&+\frac{\tau}{\beta}D^{r}_{0j}\,\mathfrak{S}_{(kir)}m_k\Big\{(\rho\tau-1)L_{ir}-\frac{\tau}{\beta}m_ib_r\Big\}-\frac{\tau^2}{2}(\rho^{}_{k}L_{ij}+\rho^{}_{i}L_{jk}-\rho^{}_{j}L_{ki})\\
&+\frac{\tau}{\beta}\Big\{\beta^{}_{k}\big((\rho\tau-1)L_{ij}-\frac{3\tau}{\beta}m_im_j\big)+\beta^{}_{i}\big((\rho\tau-1)L_{jk}-\frac{3\tau}{\beta}m_jm_k\big)\\
&-\beta^{}_{j}\big((\rho\tau-1)L_{ki}-\frac{3\tau}{\beta}m_km_i\big)\Big\},
\end{split}\end{equation}
and
\begin{equation}H^{}_{ik}=\frac{\tau^2}{\beta}(m_i\beta^{}_{k}+m^{}_{k}\beta^{}_{i})-\tau^2E_{ik}-\frac{1}{2}(G_{ik}+G_{ki})\,.\end{equation}
Again applying Lemma 3.1 to equations (3.25) and (3.26) to obtain
\begin{equation}D^{j}_{ik}=l^j\Big\{\frac{1}{\tau}\Big(\frac{2}{\beta}b^{2}-\frac{\rho}{L}\Big)^{-1}H^{}_{\beta ik}+\frac{1}{\tau}H_{ik}\Big\}-\frac{2\,m^j}{2-\rho\tau}\Big(\frac{2}{\beta}b^{2}-\frac{\rho}{L}\Big)^{-1}H^{}_{\beta ik}+\frac{L}{2\tau-\rho\tau^2}H^{\,j}_{ik}\,,\end{equation}
where we put $H^{\,j}_{ik}=g^{jm}H^{}_{mik}$. This completes the Proposition 3.1.

We now propose a lemma:
\begin{lemma}\label{lem2} If the \textsl{h}-vector is gradient then the scalar $\rho$ is constant.\end{lemma}
\begin{proof}Taking \textsl{h}-covariant derivative of (1.2) and using $L_{|k}=0$ and $h_{ij|k}=0$, we get
\begin{equation*}L\,(\dot{\partial_{j}}b_{i})_{|k}=\rho^{}_{|k}\,h_{ij}\,.\end{equation*}
Utilizing the commutation formula exhibited by
\[\dot{\partial}_{k}(T^{\,i}_{j|h})-(\dot{\partial}_{k}\,T^{\,i}_{j})^{}_{|h}=T^{\,r}_{j}\,\dot{\partial}_{k}F^{i}_{rh}-T^{\,i}_{r}\,\dot{\partial}_{k}F^{r}_{jh}-(\dot{\partial}_rT^{\,i}_{j})C^{\,r}_{hk|0}\,;\]
we get
\[ 2L\,\dot{\partial_{j}}F_{ik}=\rho^{}_{|k}\,h_{ij}-\rho^{}_{|i}\,h_{jk}\,.\]
If $b_i$ is a gradient vector, i.e. $2F^{}_{ij}=b^{}_{i|j}-b^{}_{j|i}=0$. Then above equation becomes
\begin{equation*} \rho^{}_{|k}\,h_{ij}-\rho^{}_{|i}\,h_{jk}=0\end{equation*}
which after contraction by $y^{k}$ gives $\rho^{}_{|k}\,y^{k}=0$. Differentiating $\rho^{}_{|k}\,y^{k}=0$ partially with respect to $y^{j}$, and using the commutation formula $\dot{\partial}_{j}(\rho^{}_{|k})-(\dot{\partial}_{j}\rho)_{|k}=-(\dot{\partial}_r\rho)C^{\,r}_{jk|0}$ and the fact that $\rho$ is a function of position only, we get $\rho^{}_{|j}=0$ and therefore $\partial_{j}\rho=0$. This completes the proof.
\end{proof}

Now, we find the condition for which the Cartan connection coefficients for both spaces $F^n$ and $^{*}F^n$ are the same, i.e. $^{*}F^{\,i}_{jk}=F^{\,i}_{jk}$ then $D^{i}_{jk}=0$. Therefore (3.15) and (3.16) gives $E^{}_{i0}=F^{}_{i0}$. This will give
\begin{equation}b^{}_{0|i}=0\,,\end{equation}
i.e. $\beta^{}_{|i}=0$. Differentiating $\beta^{}_{|i}=0$ partially with respect to $y^{j}$, and using the commutation formula $\dot{\partial}_{j}(\beta^{}_{|i})-(\dot{\partial}_{j}\beta)_{|i}=-(\dot{\partial}_r\beta)C^{\,r}_{ij|0}$, we get
\begin{equation}b^{}_{j|i}=b^{}_{r}\,C^{\,r}_{ij|0}.\end{equation}
This gives $F^{}_{ij}=0$ and then in view of Lemma \ref{lem2}, $F^{}_{ij}=0$ implies $\rho^{}_{i}=\rho^{}_{|i}=0$.\\
Taking \textsl{h}-covariant derivative of (1.1)(\textit{ii}) and using $L_{|k}=0,\,\rho^{}_{|k}=0$ and $h_{ij|k}=0$, we get
$(b_r\,C^{\,r}_{ij})^{}_{|k}=\big(\frac{\rho}{L}\,h_{ij}\big)^{}_{|k}=0$.
This gives
\begin{equation*}b^{}_{r|k}\,C^{\,r}_{ij}+b^{}_{r}\,C^{\,r}_{ij|k}=0\,.\end{equation*}
From (3.31), we get $b^{}_{r|k}=b^{}_{k|r}$, then above equation becomes
\[b^{}_{k|r}\,C^{\,r}_{ij}+b^{}_{r}\,C^{\,r}_{ij|k}=0\,.\]
Contracting by $y^k$, we get $b^{}_{0|r}\,C^{\,r}_{ij}+b^{}_{r}\,C^{\,r}_{ij|0}=0$. Using (3.30) and (3.31), this gives $b^{}_{i|j}=0$, i.e. the \textsl{h}-vector $b^{}_{i}$ is parallel with respect to the Cartan connection of $F^n$.

Conversely, if $b^{}_{i|j}=0$ then we get $E^{}_{ij}=0=F^{}_{ij}$ and $\beta^{}_{i}=\beta^{}_{|i}=b^{}_{j|i}y^j=0$. In view of Lemma \ref{lem2}, $F^{}_{ij}=0$ implies $\rho^{}_{i}=\rho^{}_{|i}=0$. Therefore from (3.19) we get $D^{i}_{00}=0$ and then $G^{}_{ij}=0$ and $G^{}_{j}=0$. This gives $D^{i}_{0j}=0$ and then $H^{}_{jik}=0$ and $H^{}_{ik}=0$. Therefore (3.29) implies $D^{i}_{jk}=0$ and then $^{*}F^{\,i}_{jk}=F^{\,i}_{jk}\,$. Thus, we have:

\begin{theorem}\label{th1} For the Kropina change with an \textsl{h}-vector, the Cartan connection coefficients for both spaces $F^n$ and $^{*}F^n$ are the same if and only if the \textsl{h}-vector $b_{i}$ is parallel with respect to the Cartan connection of $F^{n}$.
\end{theorem}

Transvecting (3.4) by $y^{j}$ and using $F^{\,i}_{jk}\,y^{j}=G^{\,i}_{k}$, we get
\begin{equation}^{*}G^{\,i}_{k}=G^{\,i}_{k}+D^{\,i}_{0k}\,.\end{equation}
Further transvecting (3.32) by $y^{k}$ and using $G^{\,i}_{k}\,y^{k}=2\,G^{\,i}$, we get
\begin{equation}2\,^{*}G^{\,i}=2\,G^{\,i}+D^{\,i}_{00}\,.\end{equation}
Differentiating (3.32) partially with respect to $y^{h}$ and using $\dot{\partial_{h}}G^{i}_{k}=G^{\,i}_{kh}$, we have
\begin{equation} ^{*}G^{\,i}_{kh}=G^{\,i}_{kh}+\dot{\partial}_{h}D^{\,i}_{0k}\,,\end{equation}
where $G^{\,i}_{kh}$ are the Berwald connection coefficients.

Now, if the \textsl{h}-vector $b^{}_{i}$ is parallel with respect to the Cartan connection of $F^n$, then by Theorem \ref{th1}, the Cartan connection coefficients for both spaces $F^n$ and $^{*}F^n$ are the same, therefore $D^{i}_{jk}=0$. Then from (3.34), we get $^{*}G^{\,i}_{kh}=G^{\,i}_{kh}$.

\noindent Thus, we have:
\begin{theorem}\label{th2} For the Kropina change with an \textsl{h}-vector, if the \textsl{h}-vector $b_{i}$ is parallel with respect to the Cartan connection of $F^{n}$. then the Berwald connection coefficients for both the spaces $F^{n}$ and $^{*}F^{n}$ are the same.
\end{theorem}

\section{Relation between Projective change and Kropina change with an \textsl{h}-vector}
~~~~We consider two Finsler spaces $F^{n}=(M^{n},L)$ and $^{*}F^{n}=(M^{n},\,^{*}L)$. If any geodesic on $F^{n}$ is also a geodesic on $^{*}F^{n}$ and the inverse is true, the change $L\rightarrow\,^{*}L$ of the metric is called \textit{projective}. A geodesic on $F^n$ is given by
\[\frac{dy^i}{dt}+2\,G^i(x,y)=\tau\,y^i\,;\qquad \tau=\frac{d^2s/dt^2}{ds/dt}.\]
The change $L\rightarrow\,^{*}L$ is a projective change if and only if there exists a scalar $P(x,y)$ which is positively homogeneous of degree one in $y^i$ and satisfies \cite{Ma92}
\begin{equation*}\,^{*}G^{\,i}(x,y)=\,G^{\,i}(x,y)+P(x,y)\,y^{\,i}\,.\end{equation*}
Now, we find condition for the Kropina change (1.3) with \textsl{h}-vector to be projective. From (3.33), it follows that the Kropina change with an \textsl{h}-vector is projective if and only if $D^{\,i}_{00}=2\,P\,y^i$. Then from (3.19), we get
\begin{equation}\begin{split}2\,P\,y^i= & l^i\,\tau\,\Big\{\Big(\frac{2}{\beta}\beta^{}_{0}m^2-2F^{}_{\beta 0}\Big)\Big(\frac{2}{\beta}b^2-\frac{\rho}{L}\Big)^{-1}-E^{}_{00}\Big\}\\
&-\frac{2\tau^2}{2-\rho\tau}\Big(\frac{2}{\beta}\beta^{}_{0}m^2-2F^{}_{\beta 0}\Big)\Big(\frac{2}{\beta}b^2-\frac{\rho}{L}\Big)^{-1}\,m^i+\frac{2L\tau}{2-\rho\tau}\Big(\frac{1}{\beta}\,\beta^{}_{0}\,m^i-F^{\,i}_{0}\Big).\end{split}\end{equation}
Contracting (4.1) by $y^{}_{i}$ and using $m^i\,y^{}_{i}=0=F^i_0\,y^{}_{i}$, we get
\[2\,P\,L^2=\tau\,\Big\{\Big(\frac{2}{\beta}\beta^{}_{0}m^2-2F^{}_{\beta 0}\Big)\Big(\frac{2}{\beta}b^2-\frac{\rho}{L}\Big)^{-1}-E^{}_{00}\Big\}\,L\,,\quad\textrm{i.e.},\]
\begin{equation}P=\frac{\tau}{2\,L}\,\Big\{\Big(\frac{2}{\beta}\beta^{}_{0}m^2-2F^{}_{\beta 0}\Big)\Big(\frac{2}{\beta}b^2-\frac{\rho}{L}\Big)^{-1}-E^{}_{00}\Big\}.\end{equation}
Putting the value of $P$ in (4.1), we get
\begin{equation*}-\frac{2\tau^2}{2-\rho\tau}\Big(\frac{2}{\beta}\beta^{}_{0}m^2-2F^{}_{\beta 0}\Big)\Big(\frac{2}{\beta}b^2-\frac{\rho}{L}\Big)^{-1}\,m^i+\frac{2L\tau}{2-\rho\tau}\Big(\frac{1}{\beta}\,\beta^{}_{0}\,m^i-F^{\,i}_{0}\Big)=0,\end{equation*}
i.e.,
%\[\tau\,m^i\,(A-\beta^{}_{0})=-L\,F^{i}_{0},\quad\textrm{i.e.},\]
\[F^{i}_{0}=\frac{\beta^{}_{0}}{\beta}\,m^i-\frac{1}{\beta}\,m^{}_{r}D^{\,r}_{00}\,m^i.\]
Transvecing by $g^{}_{ij}$, we get
\begin{equation}F^{}_{i0}=\frac{\beta^{}_{0}}{\beta}\,m^{}_{i}-\frac{1}{\beta}\,m^{}_{r}D^{\,r}_{00}\,m^{}_{i}\,.\end{equation}
Using (4.3) in (3.16), and referring $2\tau-\rho\tau^2\neq 0$, we get $L^{}_{ir}\,D^{\,r}_{00}=0$, which transvecting by $m^i$ and using $L_{ir}m^i=\frac{1}{L}\,m^{}_{r}$, we get $m^{}_{r}D^{\,r}_{00}=0$, and then (4.3) becomes
\begin{equation}F^{}_{i0}=\frac{\beta^{}_{0}}{\beta}\,m^{}_{i}\,.\end{equation}
This equation (4.4) is a necessary condition for the Kropina change with an \textsl{h}-vector to be a projective change.

Conversely, if (4.4) satisfies, then (3.16) gives
\begin{equation*}\Big\{(2\tau-\rho\tau^2)L_{ir}+\frac{2\tau^2}{\beta}m_im_r\Big\}D^{\,r}_{00}=0.\end{equation*}
Transvecting by $m^i$ and referring $\frac{(2\tau-\rho\tau^2)}{L}+\frac{2\tau^2}{\beta}m^2\neq 0$, we get $m^{}_{r}D^{\,r}_{00}=A=0$ and then (3.19) gives $D^{\,i}_{00}=-E^{}_{00}\,\tau\,l^i$. Therefore $^{*}F^{n}$ is projective to $F^{n}$. Thus, we have:
\begin{theorem} The Kropina change (1.3) with an \textsl{h}-vector is projective if and only if the condition (4.4) is satisfied.\end{theorem}

\begin{center}{\Large{\textbf{Ackowledgement}}}\end{center}
~~~~~~M. K. Gupta gratefully acknowledges the financial support provided by the University Grants Commission (UGC), Government of India through UGC-BSR Research Start-up-Grant.

\end{document}